\documentclass[12pt]{article}
\usepackage[utf8]{inputenc}
\usepackage{algorithm}
\usepackage{algpseudocode}
\usepackage{comment}
\usepackage{amsmath}
\usepackage{amsfonts}
\usepackage{amssymb}
\usepackage{amsbsy}
\usepackage{amsthm}
\usepackage{bm}
\usepackage{xcolor}
\usepackage{setspace}
\usepackage{float}
\usepackage[a4paper, total={7in, 8in}]{geometry}
\usepackage{cite} % Added for proper citation handling

\newcommand{\spa}[1]{\text{{$\mathcal{#1}$}}}

\newcommand{\des}{\ensuremath{\bm{x}}}

\newcommand{\vect}[1]{\ensuremath{\bm{#1}}}

\newcommand{\desCap}{\ensuremath{\bm{X}}}

\newtheorem{lemma}{Lemma}

\title{A Proof of the Convergence a Lipschitz Upper Confidence Bound}
\author{Gregory Keslin}
\date{}

\begin{document}
\maketitle

\section{Plausible Optima Overview}
We consider the setting in which the replications from a stochastic simulation system, indexed by a covariate $\vect{x} \in \spa{X}$, follow a normal distribution, $Y_j(\vect{x})\sim N(\mu(\vect{x}),\sigma^2(\vect{x}))$. The functions $\mu(\cdot)$ and $\sigma^2(\cdot)$ map the mean and variance of the replications for each covariate. In such a setting, two natural problems that can arise are screening and feasibility. The screening problem is identification of the covariates which maximize $\mu(\cdot)$. More specifically, in the screening problem, we define the set of acceptable solutions to be,
$$\mathcal{A}=\{x \in \spa{X}| \mu(\vect{x})=\max_{\vect{x}^\prime \in \spa{X}}\mu(\vect{x}^\prime)\}.$$
In the feasibility problem, we define the set of acceptable solutions to be,
$$\mathcal{A}=\{x \in \spa{X}|\mu(\vect{x})\geq c\},$$
for some constant $c$. In either problem, if $\spa{X}$ is large, it may be unfeasible to sample replications from all the covariates a sufficiently large amount of times. Therefore, we wish to use the replications of other covariates when assessing the acceptability of a given covariate. If $\mu$ is assumed to be Lipschitz continuous, the Lipschitz constant, $\lambda^\star$, can be used to help rule out unacceptable solutions.

Let $M_{\lambda}$ be the set of all Lipschitz continuous functions with Lipschitz constant $\lambda$. Then, if the Lipschitz constant of $\mu$, $\lambda^\star$, is known, we can assess whether any $\vect{x} \in \spa{X}$ is feasible by only considering the feasibility of $\vect{x}$ for functions in $M_{\lambda}$. Define $G_{\lambda^\star}(\vect{x})=\{\phi \in M_{\lambda^\star}|\phi(\vect{x})\geq c\}$ to be the set of functions for which $\vect{x}$ is feasible. To determine the feasibility of $\vect{x}$, a design of covariates (possibly random), $\mathcal{D}_k=\{\desCap_1,\ldots,\desCap_k\}$, is chosen. Then, for each $\desCap_i \in \mathcal{D}_k$, $n(k,\desCap_i)$ number of replications are generated from a normal distribution with mean $\mu(\desCap_i)$ and variance $\sigma^2(\desCap_i)$. Then, for any $\phi \in M_{\lambda^\star}$, if the discrepancy between the replications and $\phi$ is sufficiently large, the function $\phi$ is considered ``implausible''. If all functions $\phi \in G_{\lambda^\star}(\vect{x})$ are considered implausible, then it is unlikely that $\vect{x}$ is feasible. Therefore, $\vect{x}$ is classified as unfeasible. 

In this document, we choose the discrepancy measure to be the standardized absolute value differences between $\phi$ and the sample means of the replications. Let $\hat{\mu}=\{\hat{\mu}_1,\ldots,\hat{\mu}_k\}$ be the set of sample means of each design covariate and $\hat{\sigma}=\{\hat{\sigma}_1,\ldots,\hat{\sigma}_k\}$ be the set of sample standard deviations of each design covariate. We define the standardized discrepancy metric as,
$$d^1(\phi,\hat{\mu},\hat{\sigma})=\sum_{i=1}^k\frac{\sqrt{n(k,\des_i)}|\phi(\des_i)-\hat{\mu}_i|}{\hat{\sigma}_i}.$$
For a given discrepancy cutoff, $D_\alpha$, any $\vect{x} \in \spa{X}$ is classified as unfeasible if,
$$\inf_{\phi \in G_{\lambda^\star}(\vect{x})} d^1(\phi,\hat{\mu},\hat{\sigma})>D_\alpha.$$
Controlling the probability of incorrectly classifying $\vect{x}$ as unfeasible if $\vect{x} \in \mathcal{A}$ requires choosing $D_\alpha$ such that,
$$\sup_{\vect{x} \in \mathcal{A}} \mathbb{P}\left (\inf_{\phi \in G_{\lambda^\star}(\vect{x})} d^1(\phi,\hat{\mu},\hat{\sigma})>D_\alpha  \right )\leq \alpha.$$
However, the value on the left-hand side of the above equality requires knowledge of $\mu$ which is unknown in practice. In contrast, for any $\mu$, $d^1(\mu,\hat{\mu},\hat{\sigma})$ follows a known distribution: the sum of the absolute value of $k$ independent $t$ distributions with degrees of freedom $n(k,\des_i)-1$. Therefore, if $D_\alpha$ is the $1-\alpha$ quantile of $d^1(\mu,\hat{\mu},\hat{\sigma})$,
\begin{equation*}
    \begin{split}
        \sup_{\vect{x} \in \mathcal{A}} \mathbb{P}\left (\inf_{\phi \in G_{\lambda^\star}(\vect{x})} d^1(\phi,\hat{\mu},\hat{\sigma})>D_\alpha  \right )&\leq \sup_{\vect{x} \in \mathcal{A}} \mathbb{P}\left (d^1(\mu,\hat{\mu},\hat{\sigma})>D_\alpha  \right )\\
        & = \mathbb{P}\left (d^1(\mu,\hat{\mu},\hat{\sigma})>D_\alpha  \right ) \\
        & = \alpha.
    \end{split}
\end{equation*}
However, depending on $\mu$, $\inf_{\phi \in G_{\lambda^\star}(\vect{x})} d^1(\phi,\hat{\mu},\hat{\sigma})$ may be stochastically smaller than $d^1(\mu,\hat{\mu},\hat{\sigma})$. In such a case, $D_\alpha$ is conservative and leads to too few Type 1 errors.

\section{Lipschitz Estimation Proofs}
In the prior section, the Plausible Optima problem was discussed when the Lipschitz constant of $\mu$ is known. However, it is rare that the Lipschitz constant is known. Therefore, in such settings, it must be estimated. 

If not all covariates have been sampled, it is not possible to estimate an upper bound of the Lipschitz constant since the Lipschitz constant is a supremum across the entire covariate space. However, it is feasible to construct a $1-\alpha$ lower confidence bound of the Lipschitz constant. This document proves the consistency of this lower confidence bound. Because the lower confidence bound is consistent, it may make sense to use it as a point estimator of the Lipschitz constant. By increasing $\alpha$, the lower confidence bound can be made to overestimate the Lipschitz constant with higher probability. Therefore, in some sense, for large $\alpha$, we can upper bound the Lipschitz constant. 

Define $n(k,\vect{x})$ to be the function which maps the number of replications generated from $\vect{x}_i$ for design size $k$. We assume $n(k,\vect{x})$ is larger than 3 for all $k$ and $\vect{x}$. Fix $u\in \mathbb{R}^+$. Let $\overline{M}_\lambda=\bigcup_{0\leq \lambda^\prime \leq \lambda }M_\lambda$. Define $\overline{M}^u_\lambda=\{\phi \in \overline{M}_\lambda|\sup_{\vect{x}\in \spa{X}}|\phi(\vect{x})|\leq u\}$. Set $d^1(\phi,\hat{\mu},\hat{\sigma}^2)=\sum_{i=1}^k\sqrt{n(k,\des_i)}|\hat{\mu}_i-\phi(\des_i)|/\hat{\sigma}_i$.

We are interested in a $1-\alpha$ lower confidence bound of the Lipschitz constant of $\mu$, $\lambda^\star$. We define this confidence bound as the inversion of a hypothesis test,
$$\hat{\lambda}_\alpha=\inf_{\lambda} \text{ s.t } \inf_{\phi \in \overline{M}_\lambda} d^1(\phi,\hat{\mu},\hat{\sigma}^2) \leq D_\alpha.$$
By definition,
$$\mathbb{P}\left (\lambda^\star>\hat{\lambda}_\alpha \right )\geq 1-\alpha.$$
We wish to prove that as $k \to \infty$, for $\lambda<\lambda^\star$, the hypothesis test above has power $1$,
$$\lim_{k \to \infty} \mathbb{P}\left (\inf_{\phi \in \overline{M}_\lambda} d^1(\phi,\hat{\mu},\hat{\sigma}^2)>D_\alpha \right )=1.$$

To do this, we prove a series of lemmas. The first is a Uniform Law of Large Numbers for the smaller class of Lipschitz functions, $\overline{M}_\lambda^u$, when the number of replications generated at each covariate are bounded. We then extend the result to the class we are interested in, $\overline{M}_\lambda$, when the number of replications generated at each covariate are bounded. Finally, we drop the requirement that the number of replications generated at each covariate are bounded.

The following is a list of assumptions that we assume on the problem:

1. Distribution: Let $Y_j(\vect{x})$ be a replication generated from covariate $\vect{x}$. Then, $Y_j(\vect{x})\sim N(\mu(\vect{x}),\sigma^2(\vect{x}))$. 

2. Bounded Variances: $0<\underline{\sigma}\leq \sigma(\vect{x})\leq \overline{\sigma}$.

3. Uniform Monotone Lipschitz Continuity of $n(k,x)$: There exists $\lambda^\prime$ such that for all $k$, $\max_{\vect{x},\vect{x}^\prime}n(k,\vect{x})-n(k,\vect{x}^\prime)\leq \lambda^\prime s(\vect{x},\vect{x}^\prime)$ and for each $\vect{x}$, $n(\cdot,\vect{x})$ is a non-decreasing function.

4. Minimum Replication Count: For all $k$, $\vect{x} \in \spa{X}, n(k,\vect{x})>3$.

5. Compactness: $\spa{X}$ is compact.

6. Design Distribution: Each $\desCap_i$ is independently generated from probability measure $\pi$.

\begin{lemma}
\label{lemma: ulln}
 A Uniform Law of Large Numbers: Suppose $\sup_{k,\vect{x} \in \spa{X}} n(k,\vect{x}) \leq N$. For all $u$ and $\epsilon>0$,
 $$\lim_{k \to \infty}\mathbb{P}\left (\sup_{\phi \in \overline{M}^u_\lambda}\frac{|d^1(\phi,\hat{\mu},\hat{\sigma}^2)-\mathbb{E}\left ( d^1(\phi,\hat{\mu},\hat{\sigma}^2)\right )|}{k}>\epsilon\right )=0.$$
\end{lemma}
\begin{proof}
    By definition, $\overline{M}^u_\lambda$ is a set of equicontinuous uniformly bounded functions. Let $r(\cdot)$ be the metric induced by the sup norm of continuous functions, $r(\phi,\phi^\prime)=||\phi-\phi^\prime||_\infty=\sup_{x \in \spa{X}}|\phi(x)-\phi^\prime(x)|.$ Then, since $\spa{X}$ is compact, for the standard metric topology induced by $r(\cdot)$, by the Arzel\`a-Ascoli Theorem, $\overline{M}^u_\lambda$ is compact. 

    Notice that for all $\delta>0$, if $r(\phi,\phi^\prime)<\delta$, 
    \begin{equation*}
        \begin{split}
            \frac{\left |d^1(\phi,\hat{\mu},\hat{\sigma}^2)-d^1(\phi^\prime,\hat{\mu},\hat{\sigma}^2)\right |}{k}&=\left | \sum_{i=1}^k\frac{\sqrt{n(k,\desCap_i)}(|\hat{\mu}_i-\phi(\desCap_i)|-|\hat{\mu}_i-\phi^\prime(\desCap_i)|)}{k \hat{\sigma}_i} \right |\\
            & \leq \left | \sum_{i=1}^k\frac{\sqrt{n(k,\desCap_i)}(|\hat{\mu}_i-\phi(\desCap_i)-\hat{\mu}_i+\phi^\prime(\desCap_i)|)}{k \hat{\sigma}_i} \right |\\
            & \leq \delta \left | \sum_{i=1}^k\frac{\sqrt{n(k,\desCap_i)}}{k \hat{\sigma}_i} \right |\\
            & \leq \frac{N\delta}{\bar{\sigma}} \sum_{i=1}^k \frac{1}{k\chi_{i,n(k,\desCap_i)-1}},
        \end{split}
    \end{equation*}
    where $\chi_{i,n(k,\desCap_i)-1}^2$, when conditioned on $\desCap_i=\des_i$, is, for each $i$, an independent chi-squared random variable with $n(k,\des_i)-1$ degrees of freedom. Because $n(k,\des_i)>3$, $\mathbb{E}(1/\chi_{i,n(k,\desCap_i)-1}^2)=\mathbb{E}(1/(n(k,\desCap_i)-3))\leq 1$. Therefore, $\mathbb{E}(1/\chi_{i,n(k,\desCap_i)-1})\leq 2$. Thus, since $(1/\chi_{i,n(k,\desCap_i)-1})_{i=1}^k$ is an i.i.d sequence, by a triangular S.L.L.N., there exists $k_0$ such that,
    $$\mathbb{P}\left ( \sup_{k\geq k_0} \sum_{i=1}^k \frac{1}{k\chi_{i,n(k,\desCap_i)-1}} > 3\right )\leq \frac{\epsilon}{2}.$$
    Define the event $A=\{\sup_{k\geq k_0} \sum_{i=1}^k \frac{1}{k\chi_{i,n(k,\desCap_i)-1}} \leq 3\}$.
    
    (The remainder of the proof follows standard arguments for uniform strong laws of large numbers; see, e.g., Newey \cite{newey1991uniform}.)

    Since $\overline{M}^u_\lambda$ is compact, there exists a finite covering of $\overline{M}^u_\lambda$, $g_1,g_2,\ldots,g_t$ such that for all $\phi \in \overline{M}^u_\lambda$, there exists $g_j$ where $r(g_j,\phi)\leq \epsilon\underline{\sigma}/(9N)$. For each $g_j$, since $n(k,\des_i)\leq N$ almost surely and $||g_j||_\infty\leq u$, $\sup_k \mathbb{E}\left ((\sqrt{n(k,\des_i)}(\hat{\mu}_i-g_j(\des_i)))^2/\hat{\sigma}^2_i\right )<\infty$. Thus, by a triangular law of large numbers (see Durrett \cite{durrett2019probability}, for example), there exists $k_j$ such that for all $k>k_j$,
    $$\mathbb{P}\left (\frac{|d^1(g_j,\hat{\mu},\hat{\sigma}^2)-\mathbb{E}\left ( d^1(g_j,\hat{\mu},\hat{\sigma}^2)\right )|}{k}>\frac{\epsilon}{3}\right )<\frac{\epsilon}{2t}.$$
    
    Let $k^\star=\max_{j=1}^t k_j$. Consider $k\geq k^\star$. Fix arbitrary $\phi \in \overline{M}^u_\lambda$. Choose $g_j$ such that $r(g_j,\phi)\leq \epsilon/(9N)$. By the Triangle Inequality,
    \begin{equation*}
        \begin{split}
           \left |d^1(\phi,\hat{\mu},\hat{\sigma}^2)-\mathbb{E}\left ( d^1(\phi,\hat{\mu},\hat{\sigma}^2)\right )\right |&\leq \left |d^1(\phi,\hat{\mu},\hat{\sigma}^2)-d^1(g_j,\hat{\mu},\hat{\sigma}^2) \right | +\left | d^1(g_j,\hat{\mu},\hat{\sigma}^2)-\mathbb{E}\left ( d^1(g_j,\hat{\mu},\hat{\sigma}^2)\right ) \right |\\
           & +\left |\mathbb{E}\left ( d^1(g_j,\hat{\mu},\hat{\sigma}^2)\right )-\mathbb{E}\left ( d^1(\phi,\hat{\mu},\hat{\sigma}^2)\right ) \right |.
        \end{split}
    \end{equation*}
    From Jensen's inequality,
    \begin{equation*}
        \begin{split}
            \frac{1}{k}\left |\mathbb{E}\left ( d^1(g_j,\hat{\mu},\hat{\sigma}^2)\right )-\mathbb{E}\left ( d^1(\phi,\hat{\mu},\hat{\sigma}^2)\right ) \right |&\leq  \frac{1}{k}\mathbb{E}\left (\left | d^1(g_j,\hat{\mu},\hat{\sigma}^2)- d^1(\phi,\hat{\mu},\hat{\sigma}^2) \right |\right )\\
            &\leq \mathbb{E}\left ( \frac{N\epsilon\underline{\sigma}}{9N\underline{\sigma}}\sum_{i=1}^k\frac{1}{k\chi_{i,n(k,\des_i)-1}} \right )\\
            &\leq \frac{\epsilon}{3}.
        \end{split}
    \end{equation*}
    If event A occurs, 
    \begin{equation*}
        \begin{split}
             \frac{1}{k}\left | d^1(g_j,\hat{\mu},\hat{\sigma}^2)- d^1(\phi,\hat{\mu},\hat{\sigma}^2) \right |& \leq  \frac{N\epsilon\underline{\sigma}}{9N\underline{\sigma}}\sum_{i=1}^k\frac{1}{k\chi_{i,n(k,\des_i)-1}} \\
            & \leq \frac{\epsilon}{3}.
        \end{split}
    \end{equation*}
    Therefore, if event A occurs and $\left | d^1(g_j,\hat{\mu},\hat{\sigma}^2)-\mathbb{E}\left ( d^1(g_j,\hat{\mu},\hat{\sigma}^2)\right ) \right |/k\leq \epsilon/3$,
    \begin{equation*}
        \begin{split}
           \frac{1}{k}\left |d^1(\phi,\hat{\mu},\hat{\sigma}^2)-\mathbb{E}\left ( d^1(\phi,\hat{\mu},\hat{\sigma}^2)\right )\right |&\leq \frac{ \left |d^1(\phi,\hat{\mu},\hat{\sigma}^2)-d^1(g_j,\hat{\mu},\hat{\sigma}^2) \right |}{k} \\
           & +\frac{\left | d^1(g_j,\hat{\mu},\hat{\sigma}^2)-\mathbb{E}\left ( d^1(g_j,\hat{\mu},\hat{\sigma}^2)\right ) \right |}{k}\\
           & +\frac{\left |\mathbb{E}\left ( d^1(g_j,\hat{\mu},\hat{\sigma}^2)\right )-\mathbb{E}\left ( d^1(\phi,\hat{\mu},\hat{\sigma}^2)\right ) \right |}{k}\\
           &\leq \frac{\epsilon}{3}+\frac{\epsilon}{3}+\frac{\epsilon}{3}\\
           &= \epsilon.
        \end{split}
    \end{equation*}
    Thus, if event A occurs and $\max_{g_j}\left | d^1(g_j,\hat{\mu},\hat{\sigma}^2)-\mathbb{E}\left ( d^1(g_j,\hat{\mu},\hat{\sigma}^2)\right ) \right |/k\leq \epsilon/3$,
    $$\sup_{\phi \in \overline{M}^u_\lambda}\frac{1}{k}\left |d^1(\phi,\hat{\mu},\hat{\sigma}^2)-\mathbb{E}\left ( d^1(\phi,\hat{\mu},\hat{\sigma}^2)\right )\right | \leq \epsilon.$$
    Therefore, if $k\geq k^\star$, by the Union Bound,
    \begin{equation*}
        \begin{split}
            & \mathbb{P}\left (\sup_{\phi \in \overline{M}^u_\lambda}\frac{|d^1(\phi,\hat{\mu},\hat{\sigma}^2)-\mathbb{E}\left ( d^1(\phi,\hat{\mu},\hat{\sigma}^2)\right )|}{k}>\epsilon\right )\\
            &\leq \mathbb{P}\left ( A^C \bigcup \left \{\max_{g_j} \frac{\left | d^1(g_j,\hat{\mu},\hat{\sigma}^2)-\mathbb{E}\left ( d^1(g_j,\hat{\mu},\hat{\sigma}^2)\right ) \right |}{k}> \epsilon/3 \right \} \right )\\
            & \leq \mathbb{P}\left ( A^C\bigcup  \bigcup_{j=1}^t \left \{ \frac{\left | d^1(g_j,\hat{\mu},\hat{\sigma}^2)-\mathbb{E}\left ( d^1(g_j,\hat{\mu},\hat{\sigma}^2)\right ) \right |}{k}> \epsilon/3 \right \} \right )\\
            &\leq \mathbb{P}\left (A^C \right )+\sum_{j=1}^t \mathbb{P}\left ( \frac{\left | d^1(g_j,\hat{\mu},\hat{\sigma}^2)-\mathbb{E}\left ( d^1(g_j,\hat{\mu},\hat{\sigma}^2)\right ) \right |}{k}> \epsilon/3 \right )\\
            &=\frac{\epsilon}{2}+\sum_{j=1}^t \frac{\epsilon}{2t} \\
            &= \epsilon.
        \end{split}
    \end{equation*}
\end{proof}

\begin{lemma}
\label{lemma: deltaGapExp}
    Let $\mu(\cdot)$ be the true mean function and suppose $\inf_{\phi \in \overline{M}_\lambda^u} \int_{\vect{x} \in \spa{X}}|\phi(\vect{x})-\mu(\vect{x})|d\pi>0$. Then there exists $\delta>0$ such that for all $\phi \in \overline{M}_\lambda^u$,
    $$\mathbb{E}\left (  \frac{d^1(\phi,\hat{\mu},\hat{\sigma}^2)}{k}\right )\geq \delta+\mathbb{E}\left (  \frac{d^1(\mu,\hat{\mu},\hat{\sigma}^2)}{k}\right ).$$
\end{lemma}
\begin{proof}
    Because $\overline{M}_\lambda^u$ is a set of functions uniformly bounded by $u$ and $\inf_{\phi \in \overline{M}_\lambda^u} \int_{\vect{x} \in \spa{X}}|\phi(\vect{x})-\mu(\vect{x})|d\pi>0$, there exists $\epsilon>0$ such that for all $\phi \in \overline{M}_\lambda^u$,
    $$\mathbb{P}\left ( |\mu(\desCap_i)-\phi(\desCap_i)|>\epsilon \right )>\epsilon.$$
    Fix some $\des_i \in \spa{X}$ such that $|\mu(\des_i)-\phi(\des_i)|>\epsilon$. Because $\hat{\sigma}_i$ is independent of $\hat{\mu}_i$ conditional on $\desCap_i=\des_i$,
    \begin{equation*}
        \begin{split}
           &\mathbb{E}\left ( \frac{\sqrt{n(k,\des_i)}(|\hat{\mu}_i-\phi(\des_i)|-|\hat{\mu}_i-\mu(\des_i)|)}{\hat{\sigma}_i}\Big\vert\desCap_i=\des_i \right )\\
           &\geq   \left (\mathbb{E}\left (|\hat{\mu}_i-\mu(\des_i)+\mu(\des_i)-\phi(\des_i)|\right )-\mathbb{E}\left (|\hat{\mu}_i-\mu(\des_i)|\right ) \right )\mathbb{E}\left (\frac{n(k,\des_i)-1}{\sigma_i\chi_{i,n(k,\des_i)-1}}\Big\vert\desCap_i=\des_i\right ).
        \end{split}
    \end{equation*}
    It can be shown (e.g., Leone et al. \cite{leone1961folded}) that for any $\Delta \in \mathbb{R}$,
    \begin{equation*}
        \begin{split}
             \mathbb{E}\left (|\hat{\mu}_i-\mu(\des_i)+\Delta|\right )=\sigma_i\sqrt{\frac{2}{\pi}}e^{-\Delta^2/(2\sigma_i^2)}+\Delta(1-2\Phi(-\frac{\Delta}{\sigma_i})).
        \end{split}
    \end{equation*}
    Taking the derivative of the equation above, it can be shown that it is a strictly increasing continuous function of $|\Delta|$. Therefore, there exists $c>0$ such that,
    $$\sigma_i\sqrt{\frac{2}{\pi}}e^{-\epsilon^2/(2\sigma_i^2)}+\epsilon(1-2\Phi(-\frac{\epsilon}{\sigma_i}))-\sigma_i\sqrt{\frac{2}{\pi}}\geq c\overline{\sigma}.$$
    Since $|\mu(\des_i)-\phi(\des_i)|>\epsilon$,
    $$\left (\mathbb{E}\left (|\hat{\mu}_i-\mu(\des_i)+\mu(\des_i)-\phi(\des_i)|\right )-\mathbb{E}\left (|\hat{\mu}_i-\mu(\des_i)|\right ) \right )\geq c\overline{\sigma}.$$
    Furthermore,
    $$\mathbb{E}\left (\frac{n(k,\des_i)-1}{\sigma_i\chi_{i,n(k,\des_i)-1}}\Big\vert\desCap_i=\des_i\right )=\frac{n(k,\des_i)-1}{\sigma_i\sqrt{2}}\frac{\Gamma(n(k,\des_i)/2-1)}{\Gamma(n(k,\des_i)/2-1/2)}.$$
    Since $n(k,\des_i)>3$, by Gautschi's Inequality \cite{gautschi1959some},
    $$\frac{n(k,\des_i)-1}{\sigma_i\sqrt{2}}\frac{\Gamma(n(k,\des_i)/2-1)}{\Gamma(n(k,\des_i)/2-1/2)}\geq \frac{1}{\overline{\sigma}}.$$
    Therefore,
    \begin{equation*}
        \begin{split}
           &\mathbb{E}\left ( \frac{\sqrt{n(k,\desCap_i)}(|\hat{\mu}_i-\phi(\des_i)|-|\hat{\mu}_i-\mu(\des_i)|)}{\hat{\sigma}_i}\Big\vert\desCap_i=\des_i \right ) \geq c.
        \end{split}
    \end{equation*}
    We now consider the unconditioned form of the expectation above,
    \begin{equation*}
        \begin{split}
            &\mathbb{E}\left ( \frac{\sqrt{n(k,\desCap_i)}(|\hat{\mu}_i-\phi(\des_i)|-|\hat{\mu}_i-\mu(\des_i)|)}{\hat{\sigma}_i} \right )\\
            &=\mathbb{E}\left ( \frac{\sqrt{n(k,\desCap_i)}(|\hat{\mu}_i-\phi(\des_i)|-|\hat{\mu}_i-\mu(\des_i)|)}{\hat{\sigma}_i}\Big \vert |\phi(\desCap_i)-\mu(\desCap_i)|\leq \epsilon \right )\mathbb{P}\left (|\phi(\desCap_i)-\mu(\desCap_i)|\leq \epsilon \right )\\
            &+\mathbb{E}\left ( \frac{\sqrt{n(k,\desCap_i)}(|\hat{\mu}_i-\phi(\des_i)|-|\hat{\mu}_i-\mu(\des_i)|)}{\hat{\sigma}_i} \Big \vert |\phi(\desCap_i)-\mu(\desCap_i)|>\epsilon \right )\mathbb{P}\left (|\phi(\desCap_i)-\mu(\desCap_i)|>\epsilon \right ).
        \end{split}
    \end{equation*}
    Because $\hat{\mu}_i$ is normally distributed, the mean and median are the same. Thus, for all $\des_i \in \spa{X}$, $\mathbb{E}\left (|\hat{\mu}_i-\phi(\des_i)|\right )>\mathbb{E}\left (|\hat{\mu}_i-\mu(\des_i)|\right )$. Therefore,
    $$\mathbb{E}\left ( \frac{\sqrt{n(k,\desCap_i)}(|\hat{\mu}_i-\phi(\des_i)|-|\hat{\mu}_i-\mu(\des_i)|)}{\hat{\sigma}_i}\Big \vert | \phi(\desCap_i)-\mu(\desCap_i)|\leq \epsilon \right )\mathbb{P}\left (|\phi(\desCap_i)-\mu(\desCap_i)|\leq \epsilon \right )\geq 0.$$
    Therefore, for all $\phi \in \overline{M}_\lambda^u$,
    \begin{equation*}
        \begin{split}
            &\mathbb{E}\left ( \frac{\sqrt{n(k,\desCap_i)}(|\hat{\mu}_i-\phi(\des_i)|-|\hat{\mu}_i-\mu(\des_i)|)}{\hat{\sigma}_i} \right )\\
            & \geq \mathbb{E}\left ( \frac{\sqrt{n(k,\desCap_i)}(|\hat{\mu}_i-\phi(\des_i)|-|\hat{\mu}_i-\mu(\des_i)|)}{\hat{\sigma}_i} \Big \vert |\phi(\desCap_i)-\mu(\desCap_i)|>\epsilon \right )\mathbb{P}\left (|\phi(\desCap_i)-\mu(\desCap_i)|>\epsilon \right ) \\
            & \geq c\cdot \mathbb{P}\left (|\phi(\desCap_i)-\mu(\desCap_i)|>\epsilon \right )\\
            & \geq c \epsilon.
        \end{split}
    \end{equation*}
    Setting $\delta=c\epsilon$ completes the proof. 
\end{proof}

\begin{lemma}
\label{lemma: unboundedFiniteN}
    Suppose $n(k,\vect{x})\leq N$ for all $k$, $\vect{x} \in \spa{X}$ and $\inf_{\phi \in \overline{M}_\lambda} \int_{\vect{x} \in \spa{X}}|\phi(\vect{x})-\mu(\vect{x})|d\pi>0$. Then there exists $\epsilon>0$ such that,
    $$\lim_{k \to \infty}\mathbb{P}\left (\inf_{\phi \in \overline{M}_\lambda} \frac{d^1(\phi,\hat{\mu},\hat{\sigma}^2)}{k}-\mathbb{E}\left (\frac{d^1(\mu,\hat{\mu},\hat{\sigma}^2)}{k}\right )<\epsilon\right )=0.$$
\end{lemma}
\begin{proof}
    Fix $\delta>0$. Because $\spa{X}$ is compact, there exists $c>0$ such that $\max_{\vect{x},\vect{x}^\prime} s(\vect{x},\vect{x}^\prime)<c$. Let $\Delta=||\mu||_\infty$ and set $u=(1+\gamma)\Delta+\lambda c$ for some $\gamma>0$. Define $Q_\lambda^u=\overline{M}_\lambda \setminus \overline{M}_\lambda^u$. For any $\phi \in Q_\lambda^u$,
    \begin{equation*}
        \begin{split}
            \min_{\vect{x}}|\phi(\vect{x})-\mu(\vect{x})| & \geq \min_{\vect{x}} |\phi(\vect{x})|-\max_{\vect{x}} |\mu(\vect{x})|\\
            & \geq u-c\lambda - \Delta \\
            & = \gamma\Delta.
        \end{split}
    \end{equation*}
    Because $\phi$ is continuous, suppose WLOG that $\phi(\vect{x})-\mu(\vect{x})>\gamma\Delta$. Then, for all $\phi \in Q_\lambda^u$,
    $$\frac{d^1(\phi,\hat{\mu},\hat{\sigma}^2)}{k}\geq \sum_{i=1}^k \frac{\sqrt{n(k,\desCap_i)}(\gamma\Delta+\mu(\desCap_i)-\hat{\mu}_i)}{k\hat{\sigma}_i}I(\hat{\mu}_i<\mu(\desCap_i)+\gamma\Delta).$$
    There exists $\gamma$ large enough such that,
    $$\mathbb{E}\left ( \frac{\sqrt{n(k,\desCap_i)}(\gamma\Delta+\mu(\desCap_i)-\hat{\mu}_i)}{\hat{\sigma}_i}I(\hat{\mu}_i<\mu(\desCap_i)+\gamma\Delta) \right )>\mathbb{E}\left (\frac{d^1(\mu,\hat{\mu},\hat{\sigma}^2)}{k}\right )+2\delta.$$
    Since $\sup_{k} \mathbb{E}\left (n(k,\desCap_i)(\gamma\Delta+\mu(\desCap_i)-\hat{\mu}_i)^2/\hat{\sigma}^2_i \right )<\infty$, by a triangular law of large numbers,
    $$\lim_{k \to \infty} \mathbb{P} \left (\sum_{i=1}^k \frac{\sqrt{n(k,\desCap_i)}(\gamma\Delta+\mu(\desCap_i)-\hat{\mu}_i)}{k\hat{\sigma}_i}I(\hat{\mu}_i<\mu(\desCap_i)+\gamma\Delta)<\mathbb{E}\left (\frac{d^1(\mu,\hat{\mu},\hat{\sigma}^2)}{k}\right )+\delta\right )=0.$$
    Therefore,
    \begin{equation*}
        \begin{split}
            \lim_{k \to \infty} \mathbb{P} \left (\inf_{\phi \in Q_\lambda^u}\frac{d^1(\phi,\hat{\mu},\hat{\sigma}^2)}{k}<\mathbb{E}\left (\frac{d^1(\mu,\hat{\mu},\hat{\sigma}^2)}{k}\right )+\delta\right )=0.
        \end{split}
    \end{equation*}

    Since $\inf_{\phi \in \overline{M}_\lambda^u} \int_{\vect{x} \in \spa{X}}|\phi(\vect{x})-\mu(\vect{x})|d\pi>0$, by Lemma \ref{lemma: deltaGapExp}, there exists $\delta^\prime>0$ such that for all $\phi \in \overline{M}_\lambda^u$, 
    $$\mathbb{E}\left ( \frac{d^1(\phi,\hat{\mu},\hat{\sigma}^2)}{k}\right )\geq \delta^\prime+\mathbb{E}\left ( \frac{d^1(\mu,\hat{\mu},\hat{\sigma}^2)}{k}\right ).$$
    Therefore, since $n(k,\vect{x}_i)\leq N$ for all $k$, by Lemma \ref{lemma: ulln},
    \begin{equation*}
        \begin{split}
             &\lim_{k \to \infty} \mathbb{P} \left (\inf_{\phi \in \overline{M}_\lambda^u}\frac{d^1(\phi,\hat{\mu},\hat{\sigma}^2)}{k}<\mathbb{E}\left (\frac{d^1(\mu,\hat{\mu},\hat{\sigma}^2)}{k}\right )+\frac{\delta^\prime}{2}\right )\\
             &\leq \lim_{k \to \infty}\mathbb{P}\left (\sup_{\phi \in \overline{M}^u_\lambda}\frac{|d^1(\phi,\hat{\mu},\hat{\sigma}^2)-\mathbb{E}\left ( d^1(\phi,\hat{\mu},\hat{\sigma}^2)\right )|}{k}>\frac{\delta^\prime}{2}\right )\\
             & =0.
        \end{split}
    \end{equation*}
    Because $\overline{M}_\lambda=\overline{M}_\lambda^u \cup Q_\lambda^u$, for $\epsilon=\min(\delta,\delta^\prime/2)$,
    $$\lim_{k \to \infty}\mathbb{P}\left (\inf_{\phi \in \overline{M}_\lambda} \frac{d^1(\phi,\hat{\mu},\hat{\sigma}^2)}{k}-\mathbb{E}\left (\frac{d^1(\mu,\hat{\mu},\hat{\sigma}^2)}{k}\right )<\epsilon\right )=0.$$
\end{proof}

\begin{lemma}
\label{lemma:infNProb}
    Consider the case where there exists $\vect{x} \in \spa{X}$ such that $\lim_{k \to \infty}n(k,\vect{x}) =\infty$. Then, for all $|\delta|>0$, $\epsilon>0$,
    $$\lim_{k \to \infty} \mathbb{P}\left ( \frac{\sqrt{n(k,\desCap_i)}(|\hat{\mu}_i-\mu(\desCap_i)-\delta|-|\hat{\mu}_i-\mu(\desCap_i)|)}{\hat{\sigma}_i}\leq \epsilon \right )=0.$$
\end{lemma}
\begin{proof}
    Recall that $\spa{X}$ is compact, $n(k,\cdot)$ is uniformly monotone and Lipschitz and there exists $\vect{x} \in \spa{X}$ such that $\lim_{k \to \infty}n(k,\vect{x}) =\infty$. These facts imply that $\lim_{k \to \infty}\inf_{x^\prime \in \spa{X}}n(k,\vect{x}^\prime)=\infty$. 

    Let $n^0_k=\inf_{x^\prime \in \spa{X}}n(k,\vect{x}^\prime)$, $\overline{Y}_k\sim N(0,1/n^0_k)$. Since $\sigma_i \leq \overline{\sigma}$, and $n^0_k>3$, for all $\gamma>0$,
    \begin{equation*}
        \begin{split}
         &\mathbb{P}\left ( \frac{|\hat{\mu}_i-\mu(\desCap_i)-\delta|-|\hat{\mu}_i-\mu(\desCap_i)|}{\sigma_i} <\gamma \right )\\
         &=\mathbb{E}\left (\mathbb{P}\left ( \frac{|\hat{\mu}_i-\mu(\des_i)-\delta|-|\hat{\mu}_i-\mu(\des_i)|}{\sigma_i} <\gamma \Big\vert \desCap_i=\des_i \right ) \right )\\
         &\leq \mathbb{E}\left (\mathbb{P}\left ( |\overline{Y}_k/\sigma_i-\delta/\sigma_i|-|\overline{Y}_k/\sigma_i| <\gamma \Big\vert \desCap_i=\des_i \right ) \right )\\
         &\leq \mathbb{P}\left ( |\overline{Y}_k/\underline{\sigma}-\delta/\overline{\sigma}|-|\overline{Y}_k/\underline{\sigma}| <\gamma \right ).
        \end{split}
    \end{equation*}
    Since $\lim_{k \to \infty} |\overline{Y}_k/\underline{\sigma}-\delta/\overline{\sigma}|-|\overline{Y}_k/\underline{\sigma}|\overset{p}{\to} \delta/\overline{\sigma}$,
    $$\lim_{k \to \infty}\mathbb{P}\left ( \frac{|\hat{\mu}_i-\mu(\desCap_i)-\delta|-|\hat{\mu}_i-\mu(\desCap_i)|}{\sigma_i} <\frac{\delta}{2\overline{\sigma}} \right )=0.$$
    Likewise, for $\gamma>1$,
    \begin{equation*}
        \begin{split}
          \mathbb{P}\left ( \frac{\hat{\sigma}_i}{\sigma_i}>\gamma\right )&=\mathbb{E}\left (\mathbb{P}\left ( \frac{\hat{\sigma}^2_i(n(k,\des_i)-1)}{\sigma^2_i}>\gamma^2(n(k,\des_i)-1)\Big\vert \desCap_i=\des_i\right ) \right ) \\
          & \leq \mathbb{P}\left (\frac{\chi^2_{n^0_k-1}}{n^0_k-1}>\gamma^2 \right ).
        \end{split}
    \end{equation*}
    Since $\lim_{k \to \infty}\chi^2_{n^0_k-1}/(n^0_k-1) \overset{p} \to 1$,
    $$\lim_{k \to \infty}\mathbb{P}\left ( \frac{\hat{\sigma}_i}{\sigma_i}>2\right )=0.$$
    Because $\hat{\mu}_i$ and $\hat{\sigma}_i$ are independent, it follows from above that,
    $$\lim_{k \to \infty}\mathbb{P}\left ( \frac{\sqrt{n(k,\desCap_i)}(|\hat{\mu}_i-\mu(\desCap_i)-\delta|-|\hat{\mu}_i-\mu(\desCap_i)|)}{\hat{\sigma}_i}\leq \epsilon \right )=0.$$
\end{proof}

\end{document}